\documentclass[12pt]{article}
\usepackage{amssymb,amsmath}
\input{amssym.def}\input{amssym}
\usepackage{graphicx, tikz}

\newcommand{\N}{{\mathbb N}}

\newcommand{\R}{{\mathbb R}}
\newcommand{\Z}{{\mathbb Z}}

\newtheorem{theorem}{Theorem}[section]
\newtheorem{prop}{Proposition}[section]

\newtheorem{lemma}{Lemma}[section]

\newtheorem{definition}{Definition}[section]
\newenvironment{proof}{\noindent {\bf Proof.}}{ \hfill $\Box$\\ }

\newcommand\eps{\epsilon}

\def\eps{\varepsilon}

\def\orb{\text{orb}}

\def\subjclass#1{\par\medskip
\noindent\textbf{Mathematics Subject Classification (2010):} #1}
\def\keywords#1{\par\medskip
\noindent\textbf{Keywords.} #1}

\title{On co-$\sigma$-porosity of the parameters with
dense critical orbits for skew tent maps and matching on generalized $\beta$-transformations.}
\author{Henk Bruin
\thanks{Faculty of Mathematics, University of Vienna, 
Oskar Morgensternplatz 1, 1090 Vienna, Austria; {\it henk.bruin@univie.ac.at}}
 \and Gabriella Keszthelyi
\thanks{ H-1521 Budapest, Budafoki \'ut 8, Hungary; {\it keszthelyig@gmail.com}}
}
\date{\today}

\begin{document}

\maketitle

\abstract{We prove that the critical point and the point $1$ have dense orbits for Lebesgue-a.e., parameter pairs in the two-parameter skew tent family and generalised $\beta$-transformations. As an application, we show that for the generalised $\beta$-transformation with a multinacci number as slope, there is matching (i.e., $T^n(0)=T^n(1)$ for some $n \geq 1$) for Lebesgue-a.e.\ translation parameter. 
}

\subjclass{Primary: 37E10, Secondary: 11R06, 37E05, 37E45, 37A45.}
\keywords{interval maps, skew tent map, porosity, dense orbit, $\beta$-transformation, matching}

\section{Introduction}
 
 Tent maps and $\beta$-transformations are among the simplest interval maps that exhibit topologically chaotic behaviour,
 whilst having an absolutely continuous invariant measure $\mu$ (acip), provided their slopes are greater
 than one in absolute value.
 The orbit of the critical point (for the tent map) and the orbit of $1$ (for $\beta$-transformations)
 are the most important because they delimit every other orbit.
 Several paper has been devoted to whether this orbit is dense, or even typical w.r.t.\ $\mu$ (i.e., the Birkhoff Ergodic Theorem applies to this orbit), for a prevalent set of parameters.
 Schmeling \cite{S} showed that for the standard $\beta$-transformation $x \mapsto \beta x \pmod 1$,
 the orbit of $1$ is typical w.r.t.\ the acip $\mu_\beta$
 Lebesgue almost every slope $\beta > 1$.
 His proof relies on dimension-theoretic arguments.
 Bruin \cite{B}, using inducing techniques, proved analogous results for the symmetric tent family
 $T_s(x) = \min\{sx, s(1-x)\}$, $s \in (1,2]$,
 after previous results on denseness of the critical orbit by Brucks \& Misiurewicz
 \cite{BM} and Brucks \& Buckzolich \cite{BB}, who improved
 ``almost every parameter'' to ``a co-$\sigma$-porous parameter set''.
 
\begin{definition}
 A set $A \subset \R^n$ has porosity constant $\eta$ if for every $a \in A$ and $r > 0$ there is $r' \in (0,r)$ and a
 ball $B(x;2\eta r')$ which is contained in $B(a;r') \setminus A$.
 We call $A$ porous if it has some porosity constant, and $\sigma$-porous
 if it is the countable union of porous sets $A_n$ (and hence the porosity constants $\eta_n > 0$ are allowed to tend to $0$ as $n \to \infty$), see \cite{MMPZ, Z}.
 The complement of a $\sigma$-porous set is called {\em co-$\sigma$-porous}.
\end{definition}

Porous sets are nowhere dense and have no Lebesgue density points, so
$\sigma$-porous sets in $\R^n$ are meager and have zero $n$-dimensional Lebesgue measure.
However, they can have full Hausdorff dimension.

In this paper we study the problem for two-parameter families, namely the skew tent family $T_{\alpha,\beta}$ (named by Misiurewicz \& Visinescu \cite{MV} and defined in \eqref{eq:skew tent} below)
and the generalised $\beta$-transformations $G_{\alpha,\beta}(x) = \beta x + \alpha \pmod 1$.
These generalised $\beta$-transformations were probably first studied by Parry \cite{P}, and
Faller \& Pfister \cite{FP} (using methods similar to Schmeling's) 
proved that the orbit $1$ (and in fact every $x \in [0,1]$),
is typical for the acip for Lebesgue-a.e.\ parameter pair.

The strategy consists of fixing one for the parameters, and showing denseness of the critical orbit (or orbit of $1$)
for almost every value (or in cases a co-$\sigma$-porous set) of the other parameter.
This falls slightly short of two-dimensional co-$\sigma$-porosity, which is left as an open problem.
Also, the result is weaker than \cite[Theorem 2]{FP}, for generalised $\beta$-transformations, but whereas they fix $\alpha$
and vary $\beta$, we prove it the other way around, fixing $\beta$ and varying $\alpha$.
This allows the following an application, namely that if the slope $\beta > 1$ is a multinacci numbers,
 i.e., $1+\beta+\beta^2+ \dots + \beta^{N-1} = \beta^N$
 for some $N \geq 2$, then for Lebesgue-a.e.\ translation parameter $\alpha \in [0,1]$, there is $n \geq 1$ such that $G_{\alpha,\beta}^n(0) = G^n_{\alpha,\beta}(1)$.
 This property is called {\em matching} and has been studied in \cite{BCK17}, where it is shown that matching occurs for all quadratic Pisot slopes and a set of translation parameters whose complement has Hausdorff dimension $< 1$. The tribonacci slope (i.e., $1+\beta+\beta^2 = \beta^3$) is also treated there, but the multinacci case is still open.
\\[4mm]
{\bf Acknowledgments:}
GK was supported by Hungarian National Foundation for Scientific Research, Grant No.\ K124749.
Both authors acknowledges the support of Stiftung A\"OU Project 103\"ou6.

\section{Skew tent maps}

The skew tent maps $T_{\alpha,\beta}:[0,1]\to[0,1]$ are given by
\begin{equation}\label{eq:skew tent}
T_{\alpha,\beta}(x) = \begin{cases}
                       \frac{\beta}{\alpha}\, x, & x \in [0, \alpha],\\
                           \frac{\beta}{1-\alpha}\, (1-x), & x \in [\alpha,1],\\
                      \end{cases}
\end{equation}
for $0 \leq \max\{ \alpha, 1-\alpha\} < \beta \leq 1$.

Fix $\alpha \in (0,1)$ and let $\xi_n(\beta) = T_{\alpha,\beta}^n(\alpha)$, 
where we note that $\alpha$ is the critical point of the skew tent map.
We call $\xi_n:I \to [0,1]$ a {\em branch of $\xi_n$} whenever $I$ is a maximal interval on which $\xi_n$ is monotone.
Let
$$
Q_n(\beta) := \frac{\xi'_n(\beta)}{\frac{\partial}{\partial x} T_{\alpha,\beta}^n(\alpha^-)},
$$
where $\frac{\partial}{\partial x}$ denotes the space derivative and $\alpha^-$ the left limit $\lim_{x \nearrow \alpha}$.
We can compute that $Q_1(\beta) = \frac{\alpha}{\beta}$ and $Q_2(\beta) \geq \min\{ \frac{1}{\alpha}, \frac{1}{1-\alpha} \}$.
We have the recursive formula
\begin{eqnarray*}
 Q_n(\beta) &=& \frac{ \frac{\partial}{\partial \beta} T_{\alpha,\beta}(\xi_{n-1}(\beta)) 
 + \frac{\partial}{\partial x} T_{\alpha,\beta}(\xi_{n-1}(\beta)) \xi'_{n-1}(\beta) }
 {\frac{\partial}{\partial x}  T_{\alpha,\beta}(\xi_{n-1}(\beta)) \frac{\partial}{\partial x}  T^{n-1}_{\alpha,\beta}(\alpha^-) } \\
 &=&  \frac{\frac{\partial}{\partial \beta} T_{\alpha,\beta}(\xi_{n-1}(\beta))}{\frac{\partial}{\partial x}  T_{\alpha,\beta}^n(\alpha^-)}
 + Q_{n-1}(\beta).
\end{eqnarray*}
Since $\left| \frac{\partial}{\partial x} T_{\alpha,\beta}^n(\alpha^-) \right| \geq \lambda^n$ for 
$\lambda := \min\{ \frac{\beta}{\alpha}, \frac{\beta}{1-\alpha} \} > 1$, and
$\left| \frac{\partial}{\partial \beta} T_{\alpha,\beta}(\xi_{n-1}(\beta)) \right| \leq \max\{ \frac{1}{\alpha}, \frac{1}{1-\alpha} \}$,
the sequence $(Q_n(\beta))_{n \geq 1}$ is a Cauchy sequence that converges exponentially fast
to its limit $Q$, and we can check that $Q > 0$.

\begin{lemma}\label{lem:distortion}
 There are $u, C > 0$, depending only on $\alpha \in (0,1)$, such that for every $n$ and branch $\xi_n: I \to [0,1]$ and every $\beta_1 , \beta_2 \in I$,
 $$
\left| \frac{\xi'_n(\beta_2)}{\xi'_n(\beta_1)} - 1 \right| < C e^{-u n}.
$$
\end{lemma}

\begin{proof}
By the exponential convergence of $Q_n$ together with the exponential growth of
$\frac{\partial}{\partial x} T_{\alpha,\beta}^n(\alpha^-)$, we know that $|I|$ is exponentially small
in $n$. We can assume that $\beta_1 < \beta_2$, so $\Delta\beta := \beta_2-\beta_1 \leq|I|$ is positive and also exponentially small.
In general, 
$$
\frac{\partial}{\partial x} T^n_{\alpha,\beta}(\alpha^-) = 
\beta^n \left( \prod_{j=0}^{n-1} a(\xi_j(\beta)) \right)^{-1}, 
\qquad a(x) = \begin{cases}
               \alpha, & \text{for } x < \alpha,\\
               1-\alpha, & \text{for } x > \alpha.
              \end{cases}
$$
is exponentially large because $\beta > \max\{ \alpha,1-\alpha\}$.
Therefore there are exponentially small errors $\eps_{2,n}$ and $\eps_{1,n}$ such that
\begin{eqnarray*}
 \frac{\xi'_n(\beta_2)}{\xi'_n(\beta_1)} &=& 
 \frac{ Q \frac{\partial}{\partial x} T^n_{\alpha,\beta_2}(\alpha^-) + \eps_{2,n} }
 { Q \frac{\partial}{\partial x} T^n_{\alpha,\beta_1}(\alpha^-) + \eps_{1,n} } 
 = \left( \frac{\beta_2}{\beta_1} \right)^n \frac{Q + \eps_{2,n} \beta_2^{-n}}
 {Q + \eps_{1,n} \beta_1^{-n}} \\[1mm]
 &=& \left( 1 +  \frac{\Delta\beta}{\beta_1} \right)^n \ 
  \frac{Q + \eps_{2,n} \beta_2^{-1}}
 {Q + \eps_{1,n} \beta_1^{-n}} 
 = 1 + O(n \Delta \beta).
\end{eqnarray*}
This proves the lemma.
\end{proof}

We continue to fix parameter $\alpha$.
Let $W_{n-1} = W_{n-1}(\beta)$ be te maximal neighbourhood of $\beta = T_{\alpha,\beta}(c)$ on which $T_{\alpha,\beta}^{n-1}$ 
is monotone, and $M_n = M_n(\beta) = T_{\alpha,\beta}^{n-1}(W_{n-1})$.

\begin{lemma}\label{lem:Mn}
For $n \geq 1$, there are integers $1 \leq r_n, \tilde r_n < n$ such that 
$T_{\alpha,\beta}^{n-1}(\partial W_{n-1}) = \{ T_{\alpha,\beta}^{n-r_n}(c), T^{n-\tilde r_n}(c)\}$
\end{lemma}

\begin{proof}
 Let $\beta \in W_{n-1} =: [b_{n-1}, \tilde b_n]$. By maximality of $W_n$, there is $r_n < n$ such that 
 $T^{r_n-1}(b_{n-1}) = c$ and so $T_{\alpha,\beta}^n(b_{n-1}) =  T_{\alpha,\beta}^{n-r_n}(c)$. Likewise for $\tilde b_{n-1}$ and  $\tilde r_n$.
 In fact, in terms of cutting times $\{ S_k\}_{k \geq 0}$ and co-cutting times $\{ \tilde S_l \}_{l \geq 0}$,
 $b_n = n - \max\{ S_k : S_k < n\}$ and $\tilde b_n = n - \max\{ \tilde S_l : S_l < n\}$,
 see \cite{Bruin95}. 
\end{proof}

This lemma is rather trivial, but it introduces the notation for the next lemma.
Given $n \geq 4$, let $Z_n(\beta)$ be the maximal interval containing the critical value $\beta$ such that
$T_{\alpha,\beta}^{n-1}$ is monotone on $Z_n(\beta)$. 
(If $\beta$ is the common boundary point of two such interval, choose the left one.)
Since $|\frac{\partial}{\partial x} T^n_{\alpha,\beta}|$ is exponentially large and due to Lemma~\ref{lem:distortion}, 
$Z_n(\beta)$ is exponentially small.

\begin{lemma}\label{lem:xi}
Let $\beta_n$ and $\tilde \beta_n$ be the boundary points of $Z_n(\beta)$.
Then (possibly after swapping $\beta_n$ and $\tilde \beta_n$) we have
$\xi_n(\beta_n) = \xi_{n-r_n}(\beta_n)$ and $\xi_n(\tilde \beta_n) = \xi_{n-\tilde r_n}(\tilde \beta_n)$.
Moreover the quotient
$$
q: \beta' \mapsto \frac{\xi_n(\beta') - \xi_{n-r_n}(\beta')}{\xi_{n-\tilde r_n}(\beta') - \xi_{n-r_n}(\beta')}
$$
is a monotone map from $Z_n(\beta)$ onto $[0,1]$.
\end{lemma}

\begin{proof}
By maximality of $Z_n(\beta)$, at the boundary points $\beta_n$ and $\tilde \beta_n$ of $Z_n(\beta)$, there must be
 integers $r_n$ and $\tilde r_n$ such that $\xi_{r_n}(\beta_n) = c = \xi_{\tilde r_n}(\tilde \beta_n)$.
Now as $\beta'$ moves through the interior of $Z_n(\beta)$, we have
$c \in T_{\alpha,\beta}^{r_n}(\partial W_{n-1}(\beta'))$ and $c \in T_{\alpha,\beta}^{\tilde r_n}(\partial W_{n-1}(\beta'))$.
This shows that $r_n$ and $\tilde r_n$ depend only on $Z_n(\beta)$ (that is, on $\beta$ only),
and by swapping the boundary points on $Z_n(\beta)$ if necessary, the integers $r_n$ and $\tilde r_n$ are 
the same as those in Lemma~\ref{lem:Mn}.
 Finally, $q$ is clearly continuous, and thus onto $[0,1]$.
 Since the slope of $\xi_n$ is larger than the slopes of $\xi_{n-r_n}$ and $\xi_{n-\tilde r_n}$, the quotient $q$ is indeed a monotone function.
\end{proof}

Given a point $x \in [0,1] \setminus \{ \alpha\}$, let the involution $\hat x$ be the point different from $x$ such that 
$T_{\alpha,\beta}(\hat x) = T_{\alpha,\beta}(x)$. Let $p = \frac{\beta}{1-\alpha+\beta} \in [\alpha, \beta]$ 
be the orientation reversing fixed point of
$T_{\alpha,\beta}$. Note that $T_{\alpha,\beta}^2(\alpha) < \hat p < p < T_{\alpha,\beta}(\alpha)$.

\begin{prop}\label{prop:JH}
 Let $T_{\alpha,\beta}$ be a skew tent map with $0 < \min \{ \alpha, 1-\alpha\} < \beta \leq 1$.
 Then there exists $\eta \in (0,1)$ and arbitrary small neighbourhoods $J$ of the critical point $c = \alpha$ 
 for which there are intervals $H \subset J$ with $|H| \geq \eta |J|$ and $n \in \N$ such that 
 $T^n_{\alpha,\beta}$ maps $H$ monotonically onto $[\hat p, p]$.
\end{prop}

\begin{proof}
 Let $J_0 = [\hat p, p]$ and $H_0 \subset [c,p]$ be such that $T_{\alpha,\beta}^2(H_0) = [\hat p, p]$.
 First we assume that the critical point $c$ is recurrent; the proof of the proposition is simple otherwise.
 
 We construct neighbourhoods $J_k\owns c$ with subintervals $H_k$ and $\hat H_k$ adjacent to the endpoints of $J_k$ inductively.
 We always set the ratio 
 $$
 r_k = \frac{|H_k|}{|L_k|} = \frac{|\hat H_k|}{|\hat L_k|} = \frac{|\hat H_k \cup \hat H_k|}{|J_k|},
 $$ 
 where $L_k$ is the component of $J_k \setminus \{ c \}$ containing $H_k$ and $\hat L_k$ is the other component.
 
 Suppose $J_k$ and $H_k, \hat H_k \subset J_k$ are known and assume by induction that 
 \begin{equation}\label{eq:ind}
  \orb(\partial J_k) \cap \mathring J_k = \emptyset,
 \end{equation}
where $\mathring{\ }$ denotes the interior.
 Let $m_k = \min\{ n \geq 1 : T_{\alpha,\beta}^n(c) \in J_k \setminus (H_k \cup \hat H_k)\}$, and let
 $J'_{k+1}$ be the maximal neighbourhood of $c$ such that $T_{\alpha,\beta}^{m_k}(J'_{k+1}) \subset J_k$.
 By the inductive assumption, this means that $T_{\alpha,\beta}^{m_k}(\partial J'_{k+1}) \subset \partial J_k$.
 
 We claim that there are $C, u > 0$ independently of $k$ such that
 \begin{equation}\label{eq:ratio}
  \frac{|J'_{k+1}|}{|J_k|} \leq C e^{-uk}.
 \end{equation}
This is because, for any $x \in J'_{k+1} \setminus \{ c \}$ and $J_x$ the component of $ J'_{k+1} \setminus \{ c \}$
containing $x$,
$$
\left|\frac{\partial}{\partial x}  T_{\alpha,\beta}^{m_k}(x) \right| = 
\frac{ |T_{\alpha,\beta}^{m_k}(J'_{k+1})|}{|J_x|}
\leq \frac{|J_k|}{|J'_{k+1}|}.
$$
But the derivative $\frac{\partial}{\partial x}  T_{\alpha,\beta}^{m_k}$ is exponentially large and $n_k \geq k$,
so there are $C, u > 0$ depending only on $\alpha,\beta$ such that
$\frac{|J'_{k+1}|}{|J_k|} \leq C e^{-uk}$ which is \eqref{eq:ratio}.
As a consequence, $\eta := r_0 \prod_{i=0}^{\infty} \left(1-\frac{ |J'_{i+1}|}{|J_i|}\right) > 0$.

 Now there are two cases:
 \begin{itemize}
  \item $T_{\alpha,\beta}^{m_k}(J'_{k+1}) \not\supset J'_{k+1}$.
  In this case, we set $n_k = m_k$, $J_{k+1} = J'_{k+1}$ and by \eqref{eq:ind} we have 
  \begin{equation}\label{eq:ind1}
  \orb(\partial J_{k+1}) \cap \overline{J_k}  \subset \partial J_k.
 \end{equation}
 Thus we can take $H_{k+1}$ and $\hat H_{k+1} \subset J_{k+1}$ such that $T_{\alpha,\beta}^{n_k}(H_{k+1})
 = T_{\alpha,\beta}^{n_k}(\hat H_{k+1})$ equals $H_k$ or $\hat H_k$, say it is $H_k$ and $J^+_k$ is the component
 of $J_k \setminus \{ c\}$ containing $H_k$.
 Because the branches of $T_{\alpha,\beta}^{n_k}|_{J_{k+1}}$ are affine
 \begin{equation}\label{eq:rk}
 r_{k+1} = \frac{|H_k|}{|T_{\alpha,\beta}^{n_k}(J_{k+1}) |}
 \geq \frac{|H_k|}{|L_k \cup J_{k+1}|}
 \geq r_k \frac{|L_k|}{|L_k \cup J_{k+1}|} \geq r_k \left(1-\frac{|J'_{k+1}|}{|J_k|}\right).
 \end{equation}
 \item  $T_{\alpha,\beta}^{m_k}(J'_{k+1}) \supset J'_{k+1}$. In this case choose
 $n_k = \min\{ n \geq 1 : T_{\alpha,\beta}^n(c) \in J'_{k+1}\}$ and let $J_{k+1}$ 
 be the maximal neighbourhood of $c$ such that $T_{\alpha,\beta}^{n_k}(J_{k+1}) \subset J_k$.
 By \eqref{eq:ind} and \eqref{eq:ind1}, again $\orb(\partial J_{k+1}) \cap \mathring J_{k+1} = \emptyset$
 and $T_{\alpha,\beta}^{n_k}(\partial J_{k+1}) \subset \partial J_k$.
  Thus we can take $H_{k+1}$ and $\hat H_{k+1} \subset J_{k+1}$ such that $T_{\alpha,\beta}^{n_k}(H_{k+1})
 = T_{\alpha,\beta}^{n_k}(\hat H_{k+1})$ equals $H_k$ or $\hat H_k$.
 Also here \eqref{eq:rk} can be verified in the same way.
 \end{itemize}
This concludes the inductive construction, and we have 
$$
T_{\alpha,\beta}^{N_k}(H_k) = T_{\alpha,\beta}^{N_k}(\hat H_k) = [\hat p,p]
\qquad \text{ for }
N_k := n_k+n_{k-1} + \dots + n_0 + 2.
$$
By \eqref{eq:rk} also $r_k \geq r_0 \prod_{i=0}^{k-1} \left(1-\frac{ |J'_{i+1}|}{|J_i|} \right) \geq \eta$,
so the proposition follows.
\end{proof}

\begin{lemma}\label{lem:Borel}
 The set $A = \{ (\alpha,\beta) : \orb(c) \text{ is not dense for } T_{\alpha,\beta}\}$
 is a Borel set.
\end{lemma}

\begin{proof}
Let $\{ U_j\}_{j \in \N}$ be countable basis of the topology on $[0,1]$.
Then $\xi_n^{-1}([0,1] \setminus U_j)$ is closed and
$A = \cap_j \cup_n \xi_n^{-1}([0,1] \setminus U_j)$ is Borel.
\end{proof}

\begin{theorem}\label{thm:skew-non-dense}
 The set of parameters $(\alpha,\beta)$ for which the critical point of $T_{\alpha,\beta}$ has a non-dense orbit
 has zero Lebesgue measure.
\end{theorem}

\begin{proof}
First fix $\alpha \in (0,1)$. Then we show that the set of parameters $\beta$ for which the critical point of $T_{\alpha,\beta}$ 
has a non-dense orbit is $\sigma$-porous.

Let $\{ U_j\}_{j \in \N}$ be  countable basis of the topology on $[0,1]$.
Fix $\alpha$ and let
$A_j = \{\beta : T^n_{\alpha,\beta}(c) \notin U_j \text{ for all } n \geq 1\}$.
We first look at the sets $U_j$ that contain the critical point, so $\beta \in A_j$ means that $c$ is not recurrent
for $T_{\alpha,\beta}$. 
Fix $\beta \in A_j$ and define $Z_n(\beta) \owns \beta$ to be the maximal neighbourhood of $\beta$ on
which $\xi_n$ is monotone.
Now take $n$ arbitrary such that $\xi_n(Z_n(\beta)) \owns c$. Since $T^m_{\alpha,\beta}(c) \notin U_j$ for all $m \leq n$,
$\xi_n(Z_n) \supset U_j$ and by Lemma~\ref{lem:distortion}, $|\xi_n^{-1}(U_j)|/|Z_n| \geq \frac12 |U_j|$.
Since $n$ can be taken arbitrarily large $A_j$ is porous.
Hence the set of $\beta$ such that $c$ is not recurrent $\cup_{c \in U_j} A_j$ is $\sigma$-porous.

So for the rest of the proof we can assume that $c$ is recurrent and we consider the $U_j$s that don't
contain $c$.
We call $n$ a closest approach time if $|\xi_{n+1}(\beta)-\beta| < |\xi_{m+1}(\beta)-\beta|$ for all $m < n$.
Let $n'$ such a time and pick $k$ maximal such that $J_k \owns \xi_n(\beta)$, where $J_k$ are the intervals in
Proposition~\ref{prop:JH}. Once this $k$ is fixed we can 
take the smallest closest approach time $n \leq n'$ such that $\xi_n(\beta) \in J_k$.
Then $\xi_n(Z_n(\beta)) \supset J_k \supset H_k$. Therefore $|\xi_n^{-1}(H_k)|/|\xi_n^{-1}(J_k)| \geq \eta$ and therefore
$|\xi_n^{-1}(H_k \cap T_{\alpha,\beta}^{-N_k}(U_j))| / |\xi_n^{-1}(J_k)| \geq \eta |U_j|$.
Since $n$ can be taken arbitrarily large $A_j$ is porous.
Therefore the set of $\beta$ for which $c$ is recurrent but its orbit avoids some $U_j$ is $\sigma$-porous too.

Recall that $\sigma$-porous sets have zero Lebesgue measure.
Because we are speaking of Borel sets (see Lemma~\ref{lem:Borel}), the result for all $\alpha$ follows from Fubini's Theorem.
\end{proof}

\section{Generalised $\beta$-transformations}\label{sec:beta_trans}

The generalised $\beta$-transformation $G_{\alpha,\beta}:[0,1]\to[0,1]$ is given by
$$
G_{\alpha,\beta}(x) = \beta x + \alpha \pmod 1,
$$
for $\alpha \in [0,1]$ and $\beta > 1$.
Due to the symmetry $G_{\alpha,\beta}(1-x) = 1-G_{1-(\alpha+\beta \pmod 1),\beta}(x)$ it suffices to study
only parameters $\alpha \leq (1+\lfloor\alpha+\beta\rfloor -\beta)/2$.

\begin{figure}[ht]
\begin{center}
\begin{tikzpicture}[scale=0.59]
\draw[-] (0,0) -- (9,0) -- (9,9) -- (0,9) -- (0,0); 
\draw[-, dashed] (0,0) -- (9,9);
\draw[-] (9.1, 6.8) -- (8.9, 6.8);
\draw[-, draw=blue] (3,0) -- (8,9);
\draw[-, draw=blue] (8,0) -- (9,1.8);
\draw[-, draw=blue] (0,3.6) -- (3,9);
\node at (6.75,6.75) {\small $\bullet$}; 
\node at (4,-0.65) {\small $c_1=\frac{1-\alpha}{\beta}$};  
\node at (8.85,-0.65) {\small $c_2=\frac{2-\alpha}{\beta}$};
\node at (10.4,6.8) {\small $p = \frac{1-\alpha}{\beta-1}$};
\node at (-0.4,3.6) {\small $\alpha$}; 
\node at (10.6,1.8) {\small $\alpha+\beta-2$}; 
\end{tikzpicture}
\caption{The generalised $\beta$-transformation with some important points indicated.}
\label{fig:beta-tranf}
\end{center}
\end{figure}
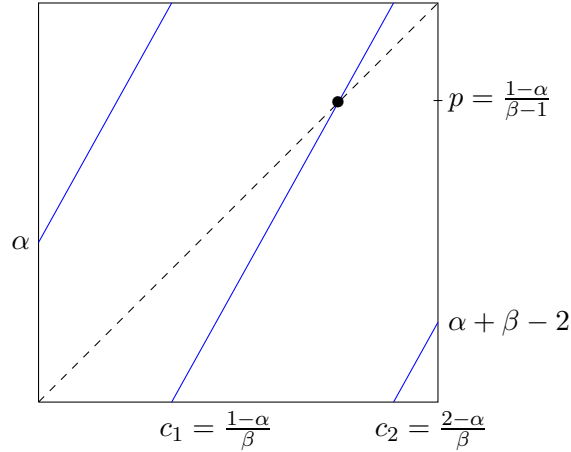 

When we consider this as map on the circle $[0,1] / \! \sim$, then it has a single discontinuity point $c = 0 = 1$. We call the left and right limit of the images  
$G_{\alpha,\beta}^k(c^-)$ and $G_{\alpha,\beta}^k(c^+)$.

\begin{lemma}\label{lem:cykel}
 For each $\beta > 1$ and $\alpha \in [0,1]$, $G_{\alpha,\beta}$ has a unique smallest invariant union $V_{\alpha,\beta}$ of non-trivial intervals.
 Furthermore, for every $\delta > 0$ there is $L$ such that
 $\bigcup_{j=0}^{L-1} G_{\alpha,\beta}^j(M) = V_{\alpha,\beta}$
 for every interval $M$ of length $|M| \geq \delta$.
\end{lemma}

\begin{proof}
 If $J$ is a non-trivial interval, that $|G_{\alpha,\beta}(J)| = \beta |J|$ unless $0$ is an interior point of $J$. Hence $J$ keeps growing under iteration of
 $G_{\alpha,\beta}$ until it contains $0$ in its interior, and in particular, there cannot be two disjoint $G_{\alpha,\beta}$-invariant unions of intervals.
 We denote the smallest such by $V_{\alpha,\beta}$.
%
 
 Now for the second statement, choose $r \in \N$ such that $\beta^r \geq 4$.
 If $\alpha = 0$, then $G_{\alpha,\beta}$ is the normal $\beta$-transformation and there is nothing to prove.
 So take $\alpha \in (0,1)$ and choose $\eps > 0$ such that $G^j_{\alpha,\beta}(B(0;\eps)) \not\owns 0$ for $0 < j < r$.
 This implies that if the interval $M \subset B(0;\eps)$ and $j \geq 1$ are
 such that $G^j_{\alpha,\beta}(M) \owns 0$,
 then $|G^j_{\alpha,\beta}(M)| \geq 2|M|$.
 For general intervals $M \subset V_{\alpha,\beta}$ of length $|M| \leq \eps$,
 we obtain $|G^j_{\alpha,\beta}(M)| \geq 2^{j/r}|M|$
 as long as $G^j_{\alpha,\beta}(M)|$ does not contain a component of $B(0;\eps) \setminus \{ 0 \}$.
 
 If $x$ is a left endpoint of (a component of) $V_{\alpha,\beta}$, then at least one point of $G_{\alpha,\beta}^{-1}(x)$ is equal to $c$ or a left endpoint of $V_{\alpha,\beta}$.
 If the latter is true for all left endpoints, then these left endpoints contain
 a periodic point, say of period $m$, which is expanding.
 Hence $G_{\alpha,\beta}^m(p+\eta) > p+\eta > p$ for all sufficiently small $\eta > 0$,
 and therefore $V_{\alpha,\beta} \setminus \bigcup_{j=0}^{m-1} [G_{\alpha,\beta}^j(p),
 G_{\alpha,\beta}^j(p+\eta))$ is forward invariant, contradicting the minimality of $V_{\alpha,\beta}$. The same argument applies to the right endpoints.
 Therefore $\partial V_{\alpha,\beta}$ is contained in the forward orbit of
 the left and right limit of the discontinuity point:
 there is $L_0 \in \N$ such
that $\partial V_{\alpha,\beta} \subset
\cup_{j=0}^{L_0} G_{\alpha,\beta}^j(\{ \alpha, \alpha+\beta \pmod 1\})$.
Since $G_{\alpha,\beta}$ is expanding, we can find $L_1 \in \N$ such that
 $$
 V_{\alpha,\beta} \subset
 \bigcup_{j=0}^{L_1} G^j_{\alpha,\beta}([0,\eps])
 \quad \text{ and } \quad
 V_{\alpha,\beta} \subset
 \bigcup_{j=0}^{L_1-1} G^j_{\alpha,\beta}([-\eps,0]).
 $$
 Then the claimed property holds for $L(\delta) = L_1 - \log_2 \delta^r$.
%
\end{proof}

In \cite{FP} it is shown that for every $\alpha \in [0,1]$ and $x \in [0,1]$,
the set of $\beta > 1$ such that $x$ is a typical point w.r.t.\ the measure of maximal entropy
(i.e., the absolutely continuous invariant probability measure (acip))
of $G_{\alpha,\beta}$ has full Lebesgue measure.
From this it follows that for Lebesgue-a.e.\ pair $(\alpha,\beta)$, the point $0$
is typical, and in particular has a dense orbit in $V_{\alpha,\beta}$ from Lemma~\ref{lem:cykel}.
This is in many ways stronger than what we will prove, but for our purposes later on,
it is important to first fix $\beta$ (and $x = 0$ but any other $x$ would work equally well)
and then vary $\alpha$.
In this way, we can use particular values of $\beta$, such as Pisot numbers.
Namely, if $\beta>1$ that are Pisot numbers, the techniques to prove this result can also be used, 
to prove that $G_{\alpha,\beta}$ has matching for a full measure set of $\alpha$, cf.\ \cite{BCK17}.
For us, only the typical denseness of the orbit of $0$ is of interest, not the stronger
property of being typical w.r.t.\ its own acip, nor shall we prove that
the set of $\alpha$ with a non-dense orbit is $\sigma$-porous.

For the generalised $\beta$-transformations, $\frac{\partial}{\partial x} G^n_{\alpha,\beta}(x) = \beta^n$ and
for fixed $\beta> 1$, 
$\xi_n(\alpha) := G^n_{\alpha,\beta}(0)$ has derivative $\xi'_n(\alpha) = \frac{\beta^n-1}{\beta-1}$.
Therefore
$$
Q_n(\alpha) := \frac{\xi'_n(\alpha)}{\frac{\partial}{\partial x} G^n_{\alpha,\beta}(0) } =
\frac{(\beta^n-1)}{\beta^n(\beta-1)} \to \frac{1}{\beta-1} \qquad \text{ as }\ n \to \infty,
$$
and we can derive the same (uniform) distortion properties for $\xi_n$ as for the tent-map case.
In particular Lemma~\ref{lem:distortion} holds.

For fixed $\beta > 1$, let $W_{n-1} = W_{n-1}(\alpha)$ be the maximal neighbourhood of $\alpha = G_{\alpha,\beta}(c^+)$
on which $T^{n-1}_{\alpha,\beta}$ is monotone.

\begin{lemma}\label{lem:Mn2}
For $n \geq 1$, there are integers $1 \leq r^+_n, r^-_n < n$ such that 
$G_{\alpha,\beta}^{n-1}(\partial W_{n-1}) = \{ G_{\alpha,\beta}^{n-r^+_n}(c^+), G_{\alpha,\beta}^{n-r^-_n}(c^-)\}$.
\end{lemma}

The proof is analogous to that of Lemma~\ref{lem:Mn} and thus omitted.
The parallel result holds for the maximal neighbourhood of $\alpha + \eta \pmod 1 = G_{\alpha,\beta}(c^-)$,
but we will not need it.

Given $n \geq 4$, let $Z_n(\alpha)$ be the maximal interval containing $\alpha$ such that
$\xi_{n-1}$ is monotone on $Z_n(\alpha)$. 
Since $|\frac{\partial}{\partial x} G^n_{\alpha,\beta}|$ is exponentially large and due to 
Lemma~\ref{lem:distortion}, $Z_n(\alpha)$ is exponentially small.
 The next lemma is the analogue of Lemma~\ref{lem:xi}, proven in the same way.

\begin{lemma}\label{lem:xi2}
Let $\alpha_n$ and $\tilde \alpha_n$ the boundary points of $Z_n(\alpha)$.
Then (after swapping $\alpha_n$ and $\tilde \alpha_n$ if necessary) we have
$\xi_n(\alpha_n) = \xi_{n-r^+_n}(\alpha_n)$ and $\xi_n(\tilde \alpha_n) = \xi_{n-r^-_n}(\tilde \alpha_n)$.
Moreover the quotient
$$
q: \alpha' \mapsto \frac{\xi_n(\alpha') - \xi_{n-r^+_n}(\alpha')}{\xi_{n-r^-_n}(\alpha') - \xi_{n-r^+_n}(\alpha')}
$$
is a monotone map from $Z_n(\alpha)$ onto $[0,1]$.
\end{lemma}

\begin{prop}\label{prop:non_dense}
For any fixed $\beta > 1$ there is $\delta \in (0,1/\beta)$ such that for Lebesgue-a.e.\ $\alpha \in [0,1]$, there is
a sequence $(n_i)$ such that $|\xi_{n_i}(Z_{n_i}(\alpha))| \geq \delta$.
\end{prop}

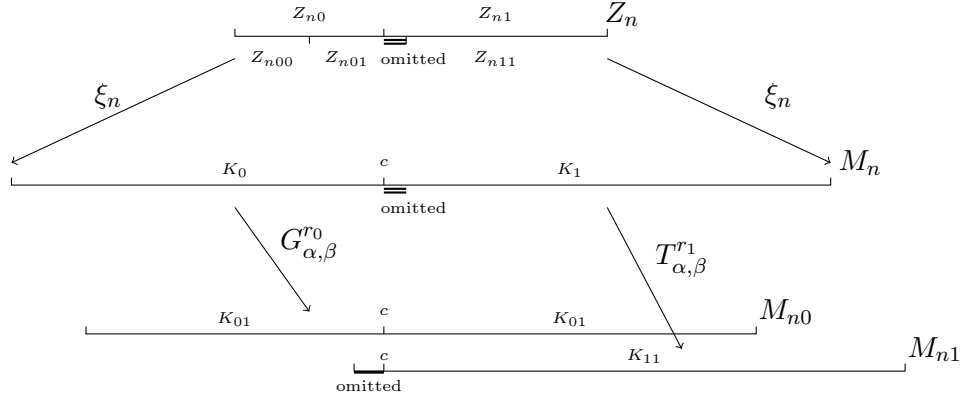
\begin{figure}[ht]
\begin{center}
\begin{tikzpicture}[scale=0.99]
\draw[-] (4,5.1) -- (4,5.0) -- (9,5) -- (9,5.1); \node at (9.2,5.3) {\small $Z_n$};  
\draw[-] (6,5.1) -- (6,5.0); \node at (5,5.3) {\tiny $Z_{n0}$}; \node at (7.5,5.3) {\tiny $Z_{n1}$};  
\draw[-] (5,4.9) -- (5,5.0); \node at (4.5,4.7) {\tiny $Z_{n00}$}; \node at (5.5,4.7) {\tiny $Z_{n01}$};  
\draw[-] (6.3,4.9) -- (6.3,5.0); 
\node at (7.5,4.7) {\tiny $Z_{n11}$}; 
\draw[-, thick] (6.0,4.9) -- (6.3,4.9); \draw[-, thick] (6.0,4.95) -- (6.3,4.95);
\node at (6.4,4.7) {\text{\tiny omitted}}; 
\draw[-, thick] (6.0,2.9) -- (6.3,2.9); \draw[-, thick] (6.0,2.95) -- (6.3,2.95);
\node at (6.4,2.7) {\text{\tiny omitted}}; 
\draw[->] (4,4.7) -- (1,3.3); \draw[->] (9,4.7) -- (12,3.3);  \node at (2.3, 4.2) {\small $\xi_n$};
 \node at (11.3, 4.2) {\small $\xi_n$};
\draw[-] (1,3.1) -- (1,3.0) -- (12,3) -- (12,3.1); \node at (12.4,3.3) {\small $M_n$};  
\draw[-] (6,3.1) -- (6,3.0); \node at (4,3.2) {\tiny $K_{0}$}; \node at (8.5,3.2) {\tiny $K_{1}$}; 
\node at (6, 3.3) {\tiny $c$};
\draw[->] (4,2.7) -- (5,1.3);   \node at (5, 2.2) {\small $G_{\alpha,\beta}^{r_0}$};
\draw[-] (2,1.1) -- (2,1.0) -- (11,1) -- (11,1.1); \node at (11.4,1.3) {\small $M_{n0}$};  
\draw[-] (6,1.1) -- (6,1.0); \node at (4,1.2) {\tiny $K_{01}$}; \node at (8.5,1.2) {\tiny $K_{01}$}; 
\node at (6, 1.3) {\tiny $c$};
\draw[->] (9,2.7) -- (10,0.8);   \node at (10, 2) {\small $T_{\alpha,\beta}^{r_1}$};
\draw[-] (5.6,0.6) -- (5.6,0.5) -- (13,0.5) -- (13,0.6); \node at (13.4,0.8) {\small $M_{n1}$};  
\draw[-] (6,0.6) -- (6,0.5);  
\node at (9.5,0.7) {\tiny $K_{11}$}; 
\node at (6, 0.7) {\tiny $c$};
\draw[-, thick] (5.6,0.48) -- (6,0.48); \draw[-, thick] (5.6,0.49) -- (6,0.49);
\node at (5.8,0.3) {\text{\tiny omitted}};
\end{tikzpicture}
\caption{Intervals used in the proof of Theorem~\ref{thm:non_dense}.}
\label{fig:proof}
\end{center}
\end{figure} 

\begin{proof}
 It suffices to show that no $\alpha \in [0,1]$ can be a density point of the set of parameters
 such that $\limsup_n |\xi_n(Z_n(\alpha))| = 0$.
 Take $\alpha_0$ and $n \in \N$ arbitrary. We will show that for a definite (i.e., independent of $n$ and $\alpha_0$)
 fraction of the set $Z_n(\alpha_0)$, there is $n'$ such that $|\xi_{n'}(Z_{n'}(\alpha)| > \delta$.
 
 First set $M_n = \xi_n(Z_n(\alpha_0))$. 
 Since $\xi'_n = \frac{\beta^n-1}{\beta-1}$ we have $|M_n| = \frac{\beta^n-1}{\beta-1} |Z_n(\alpha_0)| \geq C\beta^n  |Z_n(\alpha_0)|$
 for some $C> 0$.
 Without loss of generality we can assume that $c \in M_n$,
 and denote the two components of $M_n \setminus \{ c \}$ by $K_0$ and $K_1$, and let 
 $Z_{n0}(\alpha_0), Z_{n1}(\alpha_0) \subset Z_n(\alpha_0)$ be the subintervals such that $\xi_n(Z_{ni}(\alpha_0) = K_i$.
 If $|K_i| \leq C\beta^{n/2}|Z_n(\alpha_0)|$, then we omit $Z_{ni}(\alpha_0)$ from $Z_n$.
 Since $|K_0| + |K_1| \geq C\beta^n |Z_n(\alpha_0)|$, at most one of them can be omitted, and the omitted fraction is $\leq \beta^{-n/2}$.
 
 Next let $r_i \in \N$ be the minimal integers such that $\xi_{n+r_i}(Z_{ni}) \owns c$ and
 set $M_{ni} = \xi_{n+r_i}(Z_{ni})$ with components $K_{ij}$, $j = 0,1$, of $M_{ni} \setminus \{ c \}$, and 
 corresponding subintervals $Z_{nij}(\alpha_0) \subset Z_i(\alpha_0)$.
 Similar to the above, $M_{ni} = |K_{i0}| + |K_{i1}| = \beta^{r_i} |K_i|$
 and we omit $Z_{nij}(\alpha_0)$ if $|K_{i0}| < \beta^{-r_i/2} |K_i|$.
 Thus the relative Lebesgue measure of omitted parameters in this round is 
 $\leq \beta^{-r_1/2} :=  \min\{ \beta^{-r_0/2}, \beta^{-r_1/2}\}$.
 
 Continue inductively, until the images $M_{ni_1 \dots i_k}$ are finally longer than $\delta$.
 The non-omitted proportion is $\prod_m (1-\beta^{-r_m/2}) \geq \exp(-\sum_m \beta^{-r_m/2})$.
 Since each next $M_{ni_1 \dots i_m}$ is much larger than the previous
 $M_{ni_1 \dots i_{m-1}}$, the sequence $(r_m)_m$ is strictly decreasing and naturally
 all the (finitely many) factors in the product are $< 1$.
 Hence the proportion of non-omitted parameters is always
 at least $\exp(- \sum_m \beta^{-m/2}) = e^{-1/(\sqrt{\beta}-1)} =: \eta > 0$,
 independently of $\alpha_0$ and $n$.
 
For each non-omitted parameter $\alpha \in Z_n(\alpha_0)$ there is some $n' \leq n + r_1 + \dots + r_k$ such that
$|\xi_{n'}(Z_{n'}(\alpha))| > \delta$, and this concludes the proof.
\end{proof}

In general, we would like to have the stronger statement where $\delta = 1$. 
This is not always possible.
First, of course, the largest branch may not have length $1$.
If $0 < \alpha < \beta + \alpha < 2$, then the largest branch has length $\max\{ 1-\alpha, \alpha+\beta-1\} < 1$.
Recall from Lemma~\ref{lem:cykel} that $V_{\alpha,\beta} = \omega(x)$ for some $x \in [0,1]$.

But even so, the lack of topological mixing can prevent $\delta$ from being $1$,
see e.g.\ Theorems 4.5-4.8 and also Theorem 6.6 of \cite{OPR}.
However, with a single exception $\beta = \sqrt{2}$, $\alpha = (2-\sqrt{2})/2$, every {\bf two-branched} $G_{\alpha,\beta}$ 
is topologically mixing for $\beta \geq \sqrt{2}$.

\begin{theorem}\label{thm:non_dense}
Recall the union of intervals $V_{\alpha,\beta}$ from 
Lemma~\ref{lem:cykel}.
For every $\beta > 1$ and Lebesgue-a.e.\ $\alpha \in [0,1]$, 
the $G_{\alpha,\beta}$-orbit of $c=0$ is dense in 
$V_{\alpha,\beta}$. 
\end{theorem}

\begin{proof}
Fix $\beta > 1$ and take $\delta > 0$ as in Proposition~\ref{prop:non_dense}. 
Lemma~\ref{lem:cykel} stated that
there is $L = L(\delta) \in \N$ so that $\bigcup_{j=0}^{L-1} G_{\alpha,\beta}^j(M) = V_{\alpha,\beta}$
for every interval with diameter $|M| \geq \delta$.
Let $\{ U_k \}_k$ be a countable basis of the topology of $V_{\alpha,\beta}$.
Then $\mbox{Leb}(\bigcup_{j=0}^{L-1} G_{\alpha,\beta}^{-j}(U_k) \cap M)
\geq L^{-1} \beta^{-L} |U_k|$ for each $k$.

 By Proposition~\ref{prop:non_dense}, each neighbourhood $Z_n(\alpha_0)$ contains an $\eta$-proportion of points $\alpha$
 such that $|Z_{n'}(\alpha)| > \delta$, and therefore also for an $\eta L^{-1} \beta^{-L}  |U_k|$ proportion of 
 points $\alpha \in Z_n(\alpha_0)$, the $G_{\alpha,\beta}$-orbit of $0$ will visit $U_k$.
 Since $\alpha_0$ is not a density point of the complement, it follows that for Lebesgue full measure
 set $A_k$ of $\alpha \in [0,1]$, the $G_{\alpha,\beta}$-orbit of $0$ will visit $U_k$.
 Now take $A = \cap_k A_k$. Then $A$ has full Lebesgue measure, and the $G_{\alpha,\beta}$-orbit of $0$
 is dense in $V_{\alpha,\beta}$. 
\end{proof}

\subsection{Matching}
In this section we show how the previous result can help in proving
prevalent matching for generalised $\beta$-transformations with Pisot slopes.
We say that $G_{\alpha,\beta}$ has {\em matching} if there is an iterate $\kappa \geq 1$, called {\em matching index}
such that $G^{\kappa}_{\alpha,\beta}(0) = G^{\kappa}_{\alpha,\beta}(1)$, 
or, when viewed on the circle with discontinuity $c = 0$,
$G^{\kappa}_{\alpha,\beta}(c^-) = G^{\kappa}_{\alpha,\beta}(c^+)$.
It was shown in \cite{BCK17} that if $\beta$ is a quadratic Pisot unit, then there is matching
for Lebesgue almost every $\alpha \in [0,1]$. In fact, matching occurs on an open and dense set ({\em prevalent matching}) and
the set of parameters where matching fails has Hausdorff dimension $< 1$.

It is expected that matching is prevalent for every Pisot slope $\beta$.
Recall that $\beta > 1$ is a degree $N$ {\em Pisot unit} if it is the leading root of an irreducible polynomial
\begin{equation}\label{eq:Pisot}
 P(\beta) = \beta^N - \sum_{i=0}^{N-1} a_i \beta^i, \qquad a_i \in \Z,
\end{equation}
and all the algebraic conjugates of $\beta$ lie strictly inside the unit disk. 

The Pisot numbers we are trying to tackle are the multinacci numbers, i.e., the leading roots of
the polynomials
\begin{equation}\label{eq:Pisot_beta}
P(\beta) = \beta^N - (\beta^{N-1} + \beta^{N-2} + \dots + \beta + 1) = \beta^N - \frac{\beta^N-1}{\beta-1}.
\end{equation}
Thus $\beta < 2$ (in fact, for $N=2$, $\beta$ is the golden mean, and for $N=3$, $\beta = 1.8392867552\dots$ is the tribonacci number) and $\beta \nearrow 2$ as $N \to\infty$.
It can be easily computed that
\begin{equation}\label{eq:Pisot_beta2}
1 = \beta^{-1} + \beta^{-2} + \dots + \beta^{-N} \qquad \text{ and } \qquad 2-\beta = \beta^{-N}.
\end{equation}
In \cite{BCK17} it was shown that for $N = 3$, i.e., the tribonacci number, the non-matching 
set has Hausdorff dimension $< 1$. For all $N \geq 4$, prevalence of matching is still an open question.

Let $p = \frac{\alpha-1}{\beta-1}$ be the fixed point.
Due to symmetry, we can assume that $T(0) \leq p$, i.e., $\alpha \leq \frac{1-\alpha}{\beta-1}$ 
or equivalently $\alpha \leq \beta^{-1}$.
If $\alpha  \leq 2-\beta$, then $G_{\alpha,\beta}$ has only two branches on $[0,1]$.
In this case,  for $1 \leq n < N$, we have
$$
G_{\alpha,\beta}^n(0) = \alpha \frac{\beta^n-1}{\beta-1} \leq 
\alpha \frac{\beta^n-1}{\beta-1} + \beta^n -  \frac{\beta^n-1}{\beta-1} = G_{\alpha,\beta}^n(1),
$$
and therefore (using from \eqref{eq:Pisot_beta} that $\beta^N = \frac{\beta^N-1}{\beta-1}$) there is matching at step $N$. 

From now on, take $\alpha  > 2 - \beta$ and define
$$
d(n) := |G_{\alpha,\beta}^n(1)- G_{\alpha,\beta}^n(0)|= \sum_{i=1}^N e_i(n) \, \beta^{-i}, \qquad 
\quad e_i(n) \in \{ 0 , 1\},
$$
so
$d(1) = |\alpha+\beta-2-\alpha| = 2-\beta = \beta^{-N}$ by \eqref{eq:Pisot_beta2}.
The iteration of $d(n)$ is given by
$$
d(n+1) = \begin{cases}
            \sum_{i=1}^{N-1} e_{i+1}(n) \beta^{-i} & \text{ if this is positive};\\
            \beta^{-N} + \sum_{i=1}^{N-1} (1 - e_{i+1}(n)) \beta^{-i} & \text{ otherwise.}
         \end{cases}
$$
That is: we either shift the string $e = (e_1, \dots, e_N)$ or shift it and swap all $0$s to $1$s and vice versa.
In particular, if $e(n) = e_100\dots 0$, then $e(n+1) = 000\dots 0$ and $d(n+1)=0$, so we have matching.
This is easy to see by noting that $G_{\alpha,\beta}^n(0)$ and $G_{\alpha,\beta}^n(1)$ lie $|e_1|/\beta$ apart so their images are the same.
Therefore, if
\begin{equation}\label{eq:no_change}
G_{\alpha,\beta}^{n+i}(0) - G_{\alpha,\beta}^{n+i}(1) \text{ doesn't change sign for } 0 \leq i < N, 
\end{equation}
there is matching for some $i < N$.
Converse, if $G_{\alpha,\beta}^{n}(0) - G_{\alpha,\beta}^{n}(1)$ has just switched sign, so $e_N(n) = 1$, then matching after $N$ step 
implies \eqref{eq:no_change}.

It suffices to find an interval $U$ such that if $G_{\alpha,\beta}^m(0) \in U$,
then \eqref{eq:no_change} holds for some $n \geq m$. 
Indeed, if such $U$ exists, then 
Theorem~\ref{thm:non_dense}
implies that for a.e.\ $\alpha$, there is indeed
$m$ such that $G_{\alpha,\beta}^m(0) \in U$.
Taking this viewpoint,
we give a simpler proof of prevalence of matching than provided by \cite[Theorem 5.1]{BCK17}.

\begin{prop}\label{prop:multinacci}
The generalised $\beta$-transformation $G_{\alpha,\beta}$ with $\beta$ the tribonacci number has matching for Lebesgue-a.e.\ $\alpha \in [0,1]$. 
\end{prop}

\begin{proof}
As mentioned before, there is matching if $T$ has only two branches, so
we assume $\alpha > 2-\beta$.
If $\alpha$ is still so small that the fixed point $p = \frac{1-\alpha}{\beta-1} > 1-\beta^{-N} = \beta-1$,
and if $G_{\alpha,\beta}^n(0)$ is very close to $p$, then also \eqref{eq:no_change} holds for the next
$N$ steps, because there is no place in $[p,1]$ for $G_{\alpha,\beta}^{n+i}(1)$.
Combined with Theorem~\ref{thm:non_dense}, this means that we have almost sure matching 
for $\alpha \in [0, \beta(2-\beta)]$.

So from now on we assume that $\beta^{-1} > \alpha > \beta(2-\beta) = \beta^{1-N}$,
where the equality follows by \eqref{eq:Pisot_beta2}.
These assumptions give (recalling that $c_1 = \frac{1-\alpha}{\beta}$)
 \begin{eqnarray}\label{eq:pc1}
 \frac{1}{\beta^2} <  p-c_1 &=& \frac{1-\alpha}{\beta(\beta-1)} 
 < \frac{1-\beta^{1-N}}{\beta(\beta-1)}= \frac{\beta-1}{\beta},
   \end{eqnarray}
 where the last equality follows
  since $1-\beta^{1-N} = \beta^{1-N}(\beta-1)(\beta^{N-2} + \beta^{N-3}+ \dots + 1)
  = (\beta-1)(\beta^{-1} + \dots + \beta^{-N} - \beta^{-N})
  = (\beta-1) (1-\beta^{-N}) = (\beta-1)^2$ by \eqref{eq:Pisot_beta2}.
Therefore  
   \begin{equation}\label{eq:pc2}
    \frac{1}{\beta}-\frac{1}{\beta^2} > 
  \frac{1}{\beta} - (p-c_1) = c_2-p
  > \frac{2-\beta}{\beta} = \frac{1}{\beta^{N+1}}
   \end{equation}
by \eqref{eq:Pisot_beta2}.
Also note that 
$$
\hat p := p-\frac{1}{\beta} < c_1 = \frac{1-\alpha}{\beta} < p = \frac{1-\alpha}{\beta-1} < c_2 = \frac{2-\alpha}{\beta} < 1,
$$
so $p-c_1 < \frac1\beta$. 

 Assume that $G_{\alpha,\beta}^n(0) = p$ (the case $G_{\alpha,\beta}^n(1) = p$ goes likewise).
 Since $G_{\alpha,\beta}(0) =  \beta-1 < \frac{1-\alpha}{\beta-1} = p_1$,
 taking a finite number of iterates if necessary, we can assume that $G_{\alpha,\beta}^n(1) < G_{\alpha,\beta}^n(0)$.
 \begin{enumerate}
 \item If $d(n) = \frac{1}{\beta}$, then there is matching at the next iterate.
 \item  If $d(n) > \frac1\beta$, and therefore $G_{\alpha,\beta}^n(1) < \hat p$, then $G_{\alpha,\beta}^n(1) < G_{\alpha,\beta}^{n+1}(1) \leq G_{\alpha,\beta}^n(0)$ and $d(n+1) = \beta d(n) - 1$. 
 \item If $d(n) \leq \frac1{\beta^2}$ and therefore $d(n)\leq p-c_1$, then $c_1 < G_{\alpha,\beta}^{n}(1) \leq G_{\alpha,\beta}^n(0)$, $d(n+1) = \beta d(n)$ and $G_{\alpha,\beta}^{n+1}(1) < G_{\alpha,\beta}^{n}(1)$.
 \item The remaining case is $\frac1{\beta^2} < d(n) < \frac1\beta$. Here we have to make further case distinctions
 on $\beta$. Since $\beta$ is the tribonacci number, 
 $d(n) = \frac{1}{\beta^2} + \frac{1}{\beta^3}$ is the only possibility.
 If $p-c_1 > d(n) =  \frac{1}{\beta^2} + \frac{1}{\beta^3}$, then this case goes as part 3., and we find $G_{\alpha,\beta}^{n+1}(1) = G_{\alpha,\beta}^{n+1}(0)- \frac{1}{\beta} - \frac{1}{\beta^2}$ and by part 1.\ above, we have matching in two iterates.
 So assume that $p-c_1 \leq \frac{1}{\beta^2} + \frac{1}{\beta^3}$.
 We have
 $$
 \frac{1}{\beta^3} < 
 1-\frac{1}{\beta^2} - \frac{1}{\beta^3} \leq  c_2-p = \frac{1}{\beta} - (p_1-c_1) = \frac{1}{\beta} (1-\frac{1-\alpha}{\beta-1}) =
 \frac{1}{\beta}\frac{\alpha+\beta-2}{\beta-1}.
 $$
 We distinguish two cases:
 \begin{itemize}
  \item[(i)] $c_2-p_1 > \frac{1}{\beta^2}$ which happens when $\alpha > \frac{3\beta-\beta^2-1}{\beta}$.
  Then
  \begin{eqnarray*}
  G_{\alpha,\beta}^{n+1}(1) &=& G_{\alpha,\beta}^{n+1}(0)+\frac{1}{\beta^3} < c_2 \\
   G_{\alpha,\beta}^{n+2}(1) &=& G_{\alpha,\beta}^{n+2}(0)+\frac{1}{\beta^2} < c_2 \\
    G_{\alpha,\beta}^{n+3}(1) &=& G_{\alpha,\beta}^{n+3}(0)+\frac{1}{\beta},
 \end{eqnarray*}
 and matching occurs at the next iterate.
\item[(ii)]
$\frac{1}{\beta^3} \leq c_2-p \leq \frac{1}{\beta^2}$   which happens when 
     $\frac{\beta^2-2}{\beta^2} \leq \alpha \leq \frac{3\beta-\beta^2-1}{\beta}$.
In this case,
 \begin{eqnarray*}
  G_{\alpha,\beta}^{n+1}(1) &=& G_{\alpha,\beta}^{n+1}(0)+\frac{1}{\beta^3} < c_2 \\
   G_{\alpha,\beta}^{n+2}(1) &=& G_{\alpha,\beta}^{n+2}(0)+\frac{1}{\beta^2} > c_2 \\
    G_{\alpha,\beta}^{n+3}(1) &=& G_{\alpha,\beta}^{n+3}(0)-\frac{1}{\beta^2}-\frac{1}{\beta^3} < c_1.
 \end{eqnarray*}
Hence, if $G_{\alpha,\beta}^n(0) = p$ exactly, then $(G_{\alpha,\beta}^n+k(0), G_{\alpha,\beta}^{n+k}(1))_{k \geq 0}$ is a sequence of period $3$, and there is no matching.
However, for every $k \geq 1$, there is a small interval $V \subset p-\eps,p)$ to the left of $p$ such that $G_{\alpha,\beta}^{3k+1}(V) = V' := (p-\frac{1}{\beta^3}, p-\frac{1}{\beta^3})$.
This means that if $G_{\alpha,\beta}^n(0) \in V$,
the $G_{\alpha,\beta}^{n+3k+1}(0) \in (p-\eps,p)$,
$G_{\alpha,\beta}^{n+3k+1}(1) \in (p-\eps,p)$, so after $3k+1$ iterates, the roles of $0$ and $1$ have swapped.
By part 3.\ above, we have matching in three steps.
\end{itemize}
\end{enumerate}

\begin{figure}[ht]
\begin{center}
\begin{tikzpicture}[scale=0.9]
\node at (2.3,5) {\small \fbox{+011}}; \draw[->, draw=blue] (3.1,5) -- (4.1,5);
\node at (4.8,5) {\small \fbox{-001}}; \draw[->, draw=blue] (5.6,5) -- (6.6,5); 
\node at (3.5,5.3) {\tiny part 4(i)}; \node at (6.1,5.3) {\tiny part 4(i)}; \node at (8.5,5.3) {\tiny part 4(i)}; 
\node at (7.3,5) {\small \fbox{-010}}; \draw[->, draw=blue] (8,5) -- (9,5);
\node at (9.8,5) {\small \fbox{-100}}; \draw[->, draw=blue] (9.8,4.5) -- (10.3,1.5);
\node at (1.3,3.3) {\tiny part 4(ii)};  \node at (1.3,2.8) {\tiny $3k+1$ steps}; 
\draw[->, draw=blue] (2.3,4.5) -- (2.3,1.5); 
\draw[->, draw=blue] (3,4.5) -- (6.7,3.5); 
\node at (5.5,4.2) {\tiny $p-c_1 > d(n)$};  
\node at (2.4,1) {\small \fbox{+001}}; 
\draw[->, draw=blue] (3.1,1) -- (4.1,1); \node at (3.6,0.7) {\tiny part 3};
\node at (4.9,3) {\small \fbox{+010}}; 
\draw[->, draw=blue] (5.6,1) -- (6.6,1); \node at (6,0.7) {\tiny part 3}; 
\node at (4.9,1) {\small \fbox{+101}}; 
\draw[->, draw=blue] (4.9,2.5) -- (4.9,1.5); \node at (5.4,2) {\tiny part 2}; 
\node at (7.3,1) {\small \fbox{+100}}; 
\draw[->, draw=blue] (7.3,2.5) -- (7.3,1.5); \node at (7.8,2) {\tiny part 2}; 
\node at (7.3,3) {\small \fbox{+110}}; 
\draw[->, draw=blue] (8,1) -- (9,1); \node at (8.5,0.7) {\tiny part 1}; 
\node at (10.3,1) {\small \fbox{matching}}; 
\end{tikzpicture}
\caption{Flow-chart with codes $\pm e_1e_2e_3$, and $\pm$ indicates $\mbox{sign}(G_{\alpha,\beta}^n(0)-G_{\alpha,\beta}^n(1))$.}
\label{fig:graph1}
\end{center}
\end{figure}
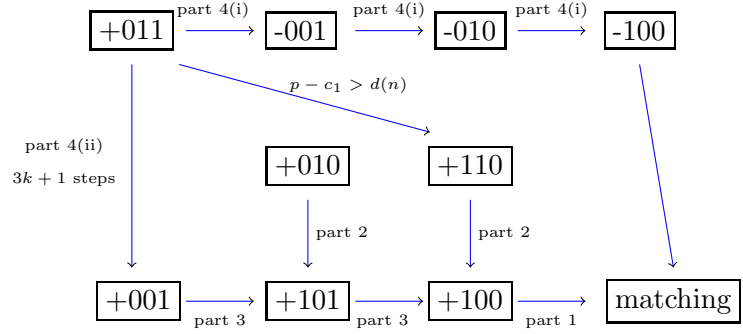 

In other words, $G_{\alpha,\beta}^n(1)$ cannot lie in the region $(\hat p, c_1)$ where $T(x) > p$, 
 and therefore, symbolically, the map $T$ acts as the shift on $e = e_1\dots e_N$.
Hence we have matching within $N$ iterates.
 
This pattern persists if $G_{\alpha,\beta}^n(0) \in U = (p-\eps,p)$ for $\eps>0$ small.  By Theorem~\ref{thm:non_dense}, 
there is matching for Lebesgue-a.e.\ $\alpha \in [0,1]$.
\end{proof}

\begin{figure}[ht]
\begin{center}
\begin{tikzpicture}[scale=0.7]
\draw[-, draw=blue] ( 15.5637019102765 , 5.39460539460539 )
--( 15.4499535256652 , 5.39460539460539 )
--( 15.3362051410539 , 5.39460539460539 )
--( 15.2224567564426 , 5.39460539460539 )
--( 15.1087083718313 , 5.39460539460539 )
--( 14.9949599872200 , 5.39460539460539 )
--( 14.8812116026087 , 5.49450549450549 )
--( 14.7674632179974 , 5.49450549450549 )
--( 14.6537148333861 , 5.49450549450549 )
--( 14.5399664487748 , 5.99400599400599 )
--( 14.4262180641635 , 3.19680319680320 )
--( 14.3124696795522 , 5.19480519480519 )
--( 14.1987212949409 , 5.19480519480519 )
--( 14.0849729103296 , 5.19480519480519 )
--( 13.9712245257183 , 5.19480519480519 )
--( 13.8574761411070 , 2.69730269730270 )
--( 13.7437277564957 , 2.69730269730270 )
--( 13.6299793718844 , 2.69730269730270 )
--( 13.5162309872731 , 2.69730269730270 )
--( 13.4024826026618 , 5.19480519480519 )
--( 13.2887342180505 , 5.19480519480519 )
--( 13.1749858334392 , 5.19480519480519 )
--( 13.0612374488279 , 5.19480519480519 )
--( 12.9474890642166 , 5.19480519480519 )
--( 12.8337406796053 , 5.19480519480519 )
--( 12.7199922949940 , 5.19480519480519 )
--( 12.6062439103827 , 5.19480519480519 )
--( 12.4924955257714 , 1.09890109890110 )
--( 12.3787471411601 , 1.09890109890110 )
--( 12.2649987565488 , 1.09890109890110 )
--( 12.1512503719375 , 1.09890109890110 )
--( 12.0375019873262 , 1.09890109890110 )
--( 11.9237536027148 , 1.09890109890110 )
--( 11.8100052181035 , 1.09890109890110 )
--( 11.6962568334922 , 1.09890109890110 )
--( 11.5825084488809 , 1.09890109890110 )
--( 11.4687600642696 , 1.09890109890110 )
--( 11.3550116796583 , 1.09890109890110 )
--( 11.2412632950470 , 1.09890109890110 )
--( 11.1275149104357 , 1.09890109890110 )
--( 11.0137665258244 , 1.09890109890110 )
--( 10.9000181412131 , 1.09890109890110 )
--( 10.7862697566018 , 1.09890109890110 )
--( 10.6725213719905 , 1.09890109890110 )
--( 10.5587729873792 , 1.09890109890110 )
--( 10.4450246027679 , 1.09890109890110 )
--( 10.3312762181566 , 1.09890109890110 )
--( 10.2175278335453 , 1.09890109890110 )
--( 10.1037794489340 , 1.09890109890110 )
--( 9.99003106432273 , 1.09890109890110 )
--( 9.87628267971142 , 1.09890109890110 )
--( 9.76253429510012 , 1.09890109890110 )
--( 9.64878591048882 , 1.09890109890110 )
--( 9.53503752587752 , 1.09890109890110 )
--( 9.42128914126622 , 1.09890109890110 )
--( 9.30754075665492 , 1.09890109890110 )
--( 9.19379237204362 , 1.09890109890110 )
--( 9.08004398743232 , 1.09890109890110 )
--( 8.96629560282102 , 1.09890109890110 )
--( 8.85254721820971 , 1.09890109890110 )
--( 8.73879883359841 , 1.09890109890110 )
--( 8.62505044898711 , 1.09890109890110 )
--( 8.51130206437581 , 1.09890109890110 )
--( 8.39755367976451 , 1.09890109890110 )
--( 8.28380529515321 , 1.09890109890110 )
--( 8.17005691054191 , 1.09890109890110 )
--( 8.05630852593060 , 0.699300699300699 )
--( 7.94256014131930 , 0.699300699300699 )
--( 7.82881175670800 , 0.699300699300699 )
--( 7.71506337209670 , 0.699300699300699 )
--( 7.60131498748540 , 0.699300699300699 )
--( 7.48756660287410 , 0.399600399600400 )
--( 7.37381821826280 , 0.399600399600400 )
--( 7.26006983365150 , 0.399600399600400 )
--( 7.14632144904019 , 0.399600399600400 )
--( 7.03257306442889 , 0.399600399600400 )
--( 6.91882467981759 , 0.399600399600400 )
--( 6.80507629520629 , 0.399600399600400 )
--( 6.69132791059499 , 0.399600399600400 )
--( 6.57757952598369 , 0.399600399600400 )
--( 6.46383114137239 , 0.399600399600400 )
--( 6.35008275676109 , 0.399600399600400 )
--( 6.23633437214978 , 0.399600399600400 )
--( 6.12258598753848 , 0.399600399600400 )
--( 6.00883760292718 , 0.399600399600400 )
--( 5.89508921831588 , 0.399600399600400 )
--( 5.78134083370458 , 0.399600399600400 )
--( 5.66759244909328 , 0.399600399600400 )
--( 5.55384406448198 , 0.399600399600400 )
--( 5.44009567987067 , 0.399600399600400 )
--( 5.32634729525937 , 0.399600399600400 )
--( 5.21259891064807 , 0.399600399600400 )
--( 5.09885052603677 , 0.399600399600400 )
--( 4.98510214142547 , 0.399600399600400 )
--( 4.87135375681417 , 0.399600399600400 )
--( 4.75760537220287 , 0.399600399600400 )
--( 4.64385698759156 , 0.399600399600400 )
--( 4.53010860298027 , 0.399600399600400 )
--( 4.41636021836896 , 0.399600399600400 )
--( 4.30261183375766 , 0.399600399600400 );

\draw[-, draw=green] ( 15.5637019102765 , 9.89010989010989 )
--( 15.4499535256652 , 5.99400599400599 )
--( 15.3362051410539 , 5.89410589410589 )
--( 15.2224567564426 , 6.79320679320679 )
--( 15.1087083718313 , 6.49350649350649 )
--( 14.9949599872200 , 5.99400599400599 )
--( 14.8812116026087 , 3.79620379620380 )
--( 14.7674632179974 , 3.19680319680320 )
--( 14.6537148333861 , 3.19680319680320 )
--( 14.5399664487748 , 3.19680319680320 )
--( 14.4262180641635 , 2.09790209790210 )
--( 14.3124696795522 , 2.09790209790210 )
--( 14.1987212949409 , 2.09790209790210 )
--( 14.0849729103296 , 2.09790209790210 )
--( 13.9712245257183 , 5.19480519480519 )
--( 13.8574761411070 , 2.69730269730270 )
--( 13.7437277564957 , 2.69730269730270 )
--( 13.6299793718844 , 2.69730269730270 )
--( 13.5162309872731 , 2.69730269730270 )
--( 13.4024826026618 , 5.19480519480519 )
--( 13.2887342180505 , 5.19480519480519 )
--( 13.1749858334392 , 5.19480519480519 )
--( 13.0612374488279 , 5.29470529470529 )
--( 12.9474890642166 , 5.19480519480519 )
--( 12.8337406796053 , 5.19480519480519 )
--( 12.7199922949940 , 5.19480519480519 )
--( 12.6062439103827 , 5.29470529470529 )
--( 12.4924955257714 , 5.19480519480519 )
--( 12.3787471411601 , 5.19480519480519 )
--( 12.2649987565488 , 5.19480519480519 )
--( 12.1512503719375 , 5.29470529470529 )
--( 12.0375019873262 , 5.29470529470529 )
--( 11.9237536027148 , 5.29470529470529 )
--( 11.8100052181035 , 5.19480519480519 )
--( 11.6962568334922 , 5.19480519480519 )
--( 11.5825084488809 , 5.99400599400599 )
--( 11.4687600642696 , 5.89410589410589 )
--( 11.3550116796583 , 6.79320679320679 )
--( 11.2412632950470 , 5.29470529470529 )
--( 11.1275149104357 , 0.499500499500500 )
--( 11.0137665258244 , 0.499500499500500 )
--( 10.9000181412131 , 0.499500499500500 )
--( 10.7862697566018 , 0.499500499500500 )
--( 10.6725213719905 , 0.499500499500500 )
--( 10.5587729873792 , 0.499500499500500 )
--( 10.4450246027679 , 0.499500499500500 )
--( 10.3312762181566 , 0.499500499500500 )
--( 10.2175278335453 , 0.499500499500500 )
--( 10.1037794489340 , 0.499500499500500 )
--( 9.99003106432273 , 0.499500499500500 )
--( 9.87628267971142 , 0.499500499500500 )
--( 9.76253429510012 , 0.499500499500500 )
--( 9.64878591048882 , 0.499500499500500 )
--( 9.53503752587752 , 0.499500499500500 )
--( 9.42128914126622 , 0.499500499500500 )
--( 9.30754075665492 , 0.499500499500500 )
--( 9.19379237204362 , 0.499500499500500 )
--( 9.08004398743232 , 0.499500499500500 )
--( 8.96629560282102 , 0.499500499500500 )
--( 8.85254721820971 , 0.499500499500500 )
--( 8.73879883359841 , 0.499500499500500 )
--( 8.62505044898711 , 0.499500499500500 )
--( 8.51130206437581 , 0.499500499500500 )
--( 8.39755367976451 , 0.499500499500500 )
--( 8.28380529515321 , 0.499500499500500 )
--( 8.17005691054191 , 0.499500499500500 )
--( 8.05630852593060 , 0.499500499500500 )
--( 7.94256014131930 , 0.499500499500500 )
--( 7.82881175670800 , 0.499500499500500 )
--( 7.71506337209670 , 0.499500499500500 )
--( 7.60131498748540 , 0.499500499500500 )
--( 7.48756660287410 , 0.499500499500500 )
--( 7.37381821826280 , 0.499500499500500 )
--( 7.26006983365150 , 0.499500499500500 )
--( 7.14632144904019 , 0.499500499500500 )
--( 7.03257306442889 , 0.499500499500500 )
--( 6.91882467981759 , 0.499500499500500 )
--( 6.80507629520629 , 0.499500499500500 )
--( 6.69132791059499 , 0.499500499500500 )
--( 6.57757952598369 , 0.499500499500500 )
--( 6.46383114137239 , 0.499500499500500 )
--( 6.35008275676109 , 0.499500499500500 )
--( 6.23633437214978 , 0.499500499500500 )
--( 6.12258598753848 , 0.499500499500500 )
--( 6.00883760292718 , 0.499500499500500 )
--( 5.89508921831588 , 0.499500499500500 )
--( 5.78134083370458 , 0.499500499500500 )
--( 5.66759244909328 , 0.499500499500500 )
--( 5.55384406448198 , 0.499500499500500 )
--( 5.44009567987067 , 0.499500499500500 )
--( 5.32634729525937 , 0.499500499500500 )
--( 5.21259891064807 , 0.499500499500500 )
--( 5.09885052603677 , 0.499500499500500 )
--( 4.98510214142547 , 0.499500499500500 )
--( 4.87135375681417 , 0.499500499500500 )
--( 4.75760537220287 , 0.499500499500500 )
--( 4.64385698759156 , 0.499500499500500 )
--( 4.53010860298027 , 0.499500499500500 )
--( 4.41636021836896 , 0.499500499500500 )
--( 4.30261183375766 , 0.499500499500500 );
\draw[-, draw=red] ( 15.5637019102765 , 0.599400599400599 )
--( 15.4499535256652 , 0.599400599400599 )
--( 15.3362051410539 , 0.599400599400599 )
--( 15.2224567564426 , 0.599400599400599 )
--( 15.1087083718313 , 0.599400599400599 )
--( 14.9949599872200 , 0.599400599400599 )
--( 14.8812116026087 , 0.599400599400599 )
--( 14.7674632179974 , 0.599400599400599 )
--( 14.6537148333861 , 0.599400599400599 )
--( 14.5399664487748 , 0.599400599400599 )
--( 14.4262180641635 , 0.599400599400599 )
--( 14.3124696795522 , 0.599400599400599 )
--( 14.1987212949409 , 0.599400599400599 )
--( 14.0849729103296 , 0.599400599400599 )
--( 13.9712245257183 , 0.599400599400599 )
--( 13.8574761411070 , 0.599400599400599 )
--( 13.7437277564957 , 0.599400599400599 )
--( 13.6299793718844 , 0.599400599400599 )
--( 13.5162309872731 , 0.599400599400599 )
--( 13.4024826026618 , 0.599400599400599 )
--( 13.2887342180505 , 0.599400599400599 )
--( 13.1749858334392 , 0.599400599400599 )
--( 13.0612374488279 , 0.599400599400599 )
--( 12.9474890642166 , 0.599400599400599 )
--( 12.8337406796053 , 0.599400599400599 )
--( 12.7199922949940 , 5.79420579420579 )
--( 12.6062439103827 , 9.79020979020979 )
--( 12.4924955257714 , 6.59340659340659 )
--( 12.3787471411601 , 6.49350649350649 )
--( 12.2649987565488 , 6.19380619380619 )
--( 12.1512503719375 , 5.79420579420579 )
--( 12.0375019873262 , 5.69430569430569 )
--( 11.9237536027148 , 6.79320679320679 )
--( 11.8100052181035 , 6.89310689310689 )
--( 11.6962568334922 , 5.99400599400599 )
--( 11.5825084488809 , 5.19480519480519 )
--( 11.4687600642696 , 5.19480519480519 )
--( 11.3550116796583 , 6.49350649350649 )
--( 11.2412632950470 , 7.19280719280719 )
--( 11.1275149104357 , 5.19480519480519 )
--( 11.0137665258244 , 5.19480519480519 )
--( 10.9000181412131 , 5.19480519480519 )
--( 10.7862697566018 , 4.79520479520480 )
--( 10.6725213719905 , 4.79520479520480 )
--( 10.5587729873792 , 4.79520479520480 )
--( 10.4450246027679 , 1.69830169830170 )
--( 10.3312762181566 , 1.69830169830170 )
--( 10.2175278335453 , 1.69830169830170 )
--( 10.1037794489340 , 1.69830169830170 )
--( 9.99003106432273 , 1.69830169830170 )
--( 9.87628267971142 , 1.69830169830170 )
--( 9.76253429510012 , 1.69830169830170 )
--( 9.64878591048882 , 1.69830169830170 )
--( 9.53503752587752 , 1.69830169830170 )
--( 9.42128914126622 , 1.69830169830170 )
--( 9.30754075665492 , 1.69830169830170 )
--( 9.19379237204362 , 1.19880119880120 )
--( 9.08004398743232 , 1.19880119880120 )
--( 8.96629560282102 , 1.19880119880120 )
--( 8.85254721820971 , 1.19880119880120 )
--( 8.73879883359841 , 1.19880119880120 )
--( 8.62505044898711 , 1.19880119880120 )
--( 8.51130206437581 , 1.19880119880120 )
--( 8.39755367976451 , 1.19880119880120 )
--( 8.28380529515321 , 1.19880119880120 )
--( 8.17005691054191 , 1.19880119880120 )
--( 8.05630852593060 , 1.19880119880120 )
--( 7.94256014131930 , 1.19880119880120 )
--( 7.82881175670800 , 1.19880119880120 )
--( 7.71506337209670 , 1.19880119880120 )
--( 7.60131498748540 , 1.19880119880120 )
--( 7.48756660287410 , 1.19880119880120 )
--( 7.37381821826280 , 1.19880119880120 )
--( 7.26006983365150 , 1.19880119880120 )
--( 7.14632144904019 , 1.19880119880120 )
--( 7.03257306442889 , 1.19880119880120 )
--( 6.91882467981759 , 1.19880119880120 )
--( 6.80507629520629 , 1.19880119880120 )
--( 6.69132791059499 , 1.19880119880120 )
--( 6.57757952598369 , 1.19880119880120 )
--( 6.46383114137239 , 1.19880119880120 )
--( 6.35008275676109 , 1.19880119880120 )
--( 6.23633437214978 , 1.19880119880120 )
--( 6.12258598753848 , 1.19880119880120 )
--( 6.00883760292718 , 1.19880119880120 )
--( 5.89508921831588 , 1.19880119880120 )
--( 5.78134083370458 , 1.19880119880120 )
--( 5.66759244909328 , 1.19880119880120 )
--( 5.55384406448198 , 1.19880119880120 )
--( 5.44009567987067 , 1.19880119880120 )
--( 5.32634729525937 , 1.19880119880120 )
--( 5.21259891064807 , 1.19880119880120 )
--( 5.09885052603677 , 1.19880119880120 )
--( 4.98510214142547 , 0.699300699300699 )
--( 4.87135375681417 , 0.699300699300699 )
--( 4.75760537220287 , 0.699300699300699 )
--( 4.64385698759156 , 0.699300699300699 )
--( 4.53010860298027 , 0.699300699300699 )
--( 4.41636021836896 , 0.699300699300699 )
--( 4.30261183375766 , 0.699300699300699 );
\draw[-, draw=black] ( 15.5637019102765 , 0.499500499500500 )
--( 15.4499535256652 , 5.49450549450549 )
--( 15.3362051410539 , 6.09390609390609 )
--( 15.2224567564426 , 1.59840159840160 )
--( 15.1087083718313 , 4.89510489510489 )
--( 14.9949599872200 , 4.89510489510489 )
--( 14.8812116026087 , 5.09490509490510 )
--( 14.7674632179974 , 6.59340659340659 )
--( 14.6537148333861 , 2.69730269730270 )
--( 14.5399664487748 , 0.899100899100899 )
--( 14.4262180641635 , 2.39760239760240 )
--( 14.3124696795522 , 4.99500499500500 )
--( 14.1987212949409 , 4.79520479520480 )
--( 14.0849729103296 , 1.79820179820180 )
--( 13.9712245257183 , 1.79820179820180 )
--( 13.8574761411070 , 1.39860139860140 )
--( 13.7437277564957 , 3.29670329670330 )
--( 13.6299793718844 , 0.799200799200799 )
--( 13.5162309872731 , 3.69630369630370 )
--( 13.4024826026618 , 2.59740259740260 )
--( 13.2887342180505 , 6.29370629370629 )
--( 13.1749858334392 , 0.899100899100899 )
--( 13.0612374488279 , 2.29770229770230 )
--( 12.9474890642166 , 0.899100899100899 )
--( 12.8337406796053 , 0.899100899100899 )
--( 12.7199922949940 , 1.29870129870130 )
--( 12.6062439103827 , 1.79820179820180 )
--( 12.4924955257714 , 3.49650349650350 )
--( 12.3787471411601 , 0.499500499500500 )
--( 12.2649987565488 , 0.499500499500500 )
--( 12.1512503719375 , 0.499500499500500 )
--( 12.0375019873262 , 0.499500499500500 )
--( 11.9237536027148 , 0.499500499500500 )
--( 11.8100052181035 , 0.499500499500500 )
--( 11.6962568334922 , 0.499500499500500 )
--( 11.5825084488809 , 0.999000999000999 )
--( 11.4687600642696 , 1.29870129870130 )
--( 11.3550116796583 , 2.89710289710290 )
--( 11.2412632950470 , 1.89810189810190 )
--( 11.1275149104357 , 1.19880119880120 )
--( 11.0137665258244 , 0.799200799200799 )
--( 10.9000181412131 , 1.09890109890110 )
--( 10.7862697566018 , 4.89510489510489 )
--( 10.6725213719905 , 0.899100899100899 )
--( 10.5587729873792 , 1.29870129870130 )
--( 10.4450246027679 , 0.999000999000999 )
--( 10.3312762181566 , 1.49850149850150 )
--( 10.2175278335453 , 1.79820179820180 )
--( 10.1037794489340 , 0.999000999000999 )
--( 9.99003106432273 , 0.799200799200799 )
--( 9.87628267971142 , 6.29370629370629 )
--( 9.76253429510012 , 1.19880119880120 )
--( 9.64878591048882 , 1.39860139860140 )
--( 9.53503752587752 , 1.19880119880120 )
--( 9.42128914126622 , 0.899100899100899 )
--( 9.30754075665492 , 1.09890109890110 )
--( 9.19379237204362 , 0.699300699300699 )
--( 9.08004398743232 , 1.39860139860140 )
--( 8.96629560282102 , 1.29870129870130 )
--( 8.85254721820971 , 1.49850149850150 )
--( 8.73879883359841 , 0.899100899100899 )
--( 8.62505044898711 , 6.79320679320679 )
--( 8.51130206437581 , 0.499500499500500 )
--( 8.39755367976451 , 0.499500499500500 )
--( 8.28380529515321 , 0.499500499500500 )
--( 8.17005691054191 , 0.499500499500500 )
--( 8.05630852593060 , 0.499500499500500 )
--( 7.94256014131930 , 0.499500499500500 )
--( 7.82881175670800 , 0.499500499500500 )
--( 7.71506337209670 , 0.499500499500500 )
--( 7.60131498748540 , 0.499500499500500 )
--( 7.48756660287410 , 6.89310689310689 )
--( 7.37381821826280 , 1.59840159840160 )
--( 7.26006983365150 , 0.999000999000999 )
--( 7.14632144904019 , 1.19880119880120 )
--( 7.03257306442889 , 0.899100899100899 )
--( 6.91882467981759 , 1.49850149850150 )
--( 6.80507629520629 , 3.29670329670330 )
--( 6.69132791059499 , 2.79720279720280 )
--( 6.57757952598369 , 1.39860139860140 )
--( 6.46383114137239 , 0.899100899100899 )
--( 6.35008275676109 , 0.499500499500500 )
--( 6.23633437214978 , 0.499500499500500 )
--( 6.12258598753848 , 0.499500499500500 )
--( 6.00883760292718 , 0.499500499500500 )
--( 5.89508921831588 , 0.499500499500500 )
--( 5.78134083370458 , 0.499500499500500 )
--( 5.66759244909328 , 0.999000999000999 )
--( 5.55384406448198 , 1.19880119880120 )
--( 5.44009567987067 , 1.29870129870130 )
--( 5.32634729525937 , 1.49850149850150 )
--( 5.21259891064807 , 0.699300699300699 )
--( 5.09885052603677 , 1.39860139860140 )
--( 4.98510214142547 , 0.899100899100899 )
--( 4.87135375681417 , 3.39660339660340 )
--( 4.75760537220287 , 1.39860139860140 )
--( 4.64385698759156 , 2.29770229770230 )
--( 4.53010860298027 , 0.999000999000999 )
--( 4.41636021836896 , 3.79620379620380 )
--( 4.30261183375766 , 0.499500499500500 );
\draw[->] (3,0) -- (16,0); \node at (16,-0.3) {$\alpha$};
\draw[-] (4.5,-0.1) -- (4.5,0.1); \node at (4.5,-0.4) {\small $0.3$};
\draw[-] (9.75,-0.1) -- (9.75,0.1); \node at (9.75,-0.4) {\small $0.4$};
\draw[-] (15,-0.1) -- (15,0.1); \node at (15,-0.4) {\small $0.5$};
\draw[->] (3,0) -- (3,10); \node at (2.7,10) {$n$};
\draw[-] (2.9,7.5) -- (3.1,7.5); \node at (2.6,7.5) {\small $75$};
\draw[-] (2.9,5) -- (3.1,5); \node at (2.6,5) {\small $50$};
\draw[-] (2.9,2.5) -- (3.1,2.5); \node at (2.6,2.5) {\small $25$};
\node[color=blue] at (7,0.6) {\tiny $0110$};
\node[color=green] at (9,0.3) {\tiny $0101$};
\node[color=red] at (6,1.4) {\tiny $0111$};
\node[color=black] at (6.5,5.4) {\tiny from start};
\node[color=black] at (6.45,5) {\tiny $T(0)$, $T(1)$};
\end{tikzpicture}
\caption{Number of iterates before matching for the tetrabonacci number.}
\label{fig:graph2}
\end{center}
\end{figure}
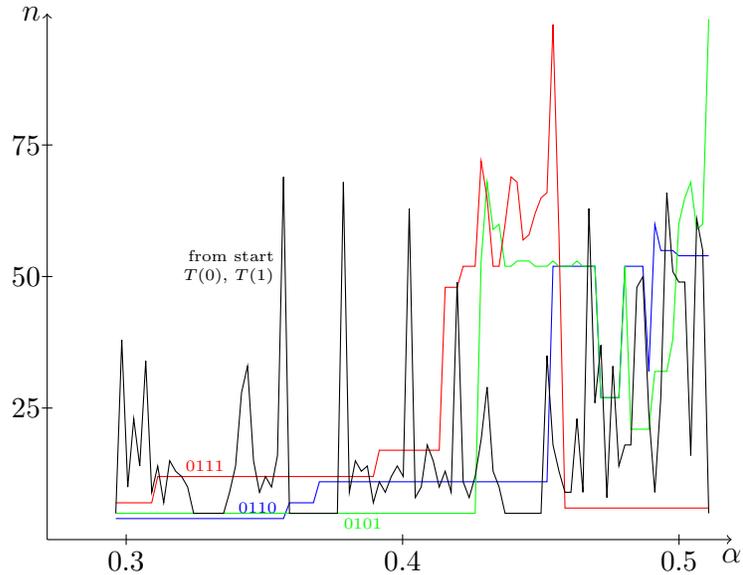 

For the tetrabonacci number (i.e., $N=4$) a similar proof seems possible, but the number of case distinctions 
becomes very large. Instead, in Figure~\ref{fig:graph2}, we give some numerics on
the number of iterates needed before matching occurs.
The black curve has starting point $(T(0), T(1))$, and the other curves are in the gist of the proof of 
Proposition~\ref{prop:multinacci}, namely they start when $G_{\alpha,\beta}^n(0$ is close to the fixed point:
 $G_{\alpha,\beta}^n(0) = p -\eps$ for $\eps = 0.01$ and $G_{\alpha,\beta}^(1) = G_{\alpha,\beta}^n(y) - d(N)$
for $d(n) = \sum_{i=1}^4 e_i(n) \beta^{-i}$ with $e = 0110, \ 0101$ and $0111$.
The range $\alpha \in [\beta^{-3}, \beta^{-1}]$ with $100$ grid-points in the horizontal direction.


\begin{thebibliography}{XXX}

\bibitem{BB} K.\ Brucks, Z.\ Buczolich,
{\em Trajectory of the turning point is dense for a co-$\sigma$-porous set of tent maps,}
Fund.\ Math.\ {\bf 165} (2000) 95--123.

\bibitem{BM} K.\  Brucks, M.\ Misiurewicz, 
{\em The trajectory of the turning point is dense for almost all tent maps.} 
Ergodic Theory Dynam.\ Systems {\bf 16} (1996), no. 6, 1173--1183. 

\bibitem{Bruin95} H.\ Bruin,
{\em Combinatorics of the kneading map,}
Internat.\ J.\ Bifur.\ Chaos Appl.\ Sci.\ Engrg.\ {\bf 5} (1995), 1339--1349.

\bibitem{B} H.\ Bruin,
{\em For almost every tent-map, the turning point is typical,}
Fund.\ Math.\ {\bf 155} (1998) 215--235.

\bibitem{BCK17}  H.\ Bruin, C.\ Carminati, C.\ Kalle,
{\em Matching for generalised $\beta$-transformations,}
Indagationes Mathematicae, {\bf 28} (2017), no. 1, 55--73. 

\bibitem{FP} B.\ Faller, C.-E.\ Pfister,  
{\em A point is normal for almost all maps $\beta x + \alpha \pmod 1$ or generalized $\beta$-transformations}, 
Ergod.\ Th.\ \& Dynam.\ Sys.\ {\bf 29} (2009), 1529--1547. 
%
%

\bibitem{G} P.\ G\'ora, 
{\em Invariant densities for generalized $\beta$-maps.}
Ergodic Theory Dynam.\ Systems {\bf 27} (2007), no. 5, 1583--1598. 

\bibitem{MMPZ}  M.\ E.\ Mera, M.\ Mor\'an, D.\ Preiss, L.\ Zaj\'{\i}\v{c}ek, 
{\em Porosity, $\sigma$-porosity and measures}, 
Nonlinearity {\bf 16} (2003), no. 1, 247--255.

\bibitem{MV} M.\ Misiurewicz, E.\ Visinescu, 
{\em Kneading sequences of skew tent maps.}
Ann. Inst. H. Poincar\'e Probab. Statist. {\bf 27} (1991), no. 1, 125--140. 

\bibitem{OPR} P.\ Oprocha, P.\ Potorski, P.\ Raith, 
{\em Mixing properties in expanding Lorenz maps}, 
Adv.\ Math.\ {\bf 343} (2019), 712--755.

\bibitem{P} W.\ Parry, 
{\em Representations for real numbers.}
Acta Math. Acad. Sci. Hungar.\ {\bf 15} (1964), 95--105. 

\bibitem{S} J.\ Schmeling,
{\em Symbolic dynamics for $\beta$-shifts and self-normal numbers,}
Ergodic Theory and Dynamical Systems, {\bf 17} (1997) 675--694.

\bibitem{Z} L.\  Zaj\'{\i}\v{c}ek, 
{\em Porosity and $\sigma$-porosity}, 
Real Anal.\ Exchange {\bf 13} (1987/88), no. 2, 314--350. 
\end{thebibliography}
\end{document}